\documentclass[a4paper,12pt]{article}
\usepackage[bookmarks=true,
              bookmarksnumbered=true, breaklinks=true,
              pdfstartview=FitH, hyperfigures=false,
              plainpages=false, naturalnames=true,
              colorlinks=true,
              pdfpagelabels]{hyperref}
\usepackage{geometry,latexsym,amssymb,amsmath,amsthm,color,bm}
\usepackage[latin5]{inputenc}
\usepackage{enumerate}
\usepackage{enumitem}
\usepackage[T1]{fontenc}
\usepackage{authblk}
\usepackage[all]{xy}
\usepackage{palatino}
\usepackage{indentfirst}
\usepackage{graphics}
\usepackage{tabularx}
\usepackage{lipsum}

\usepackage{anysize}
\marginsize{2.5cm}{2.5cm}{2.5cm}{2.5cm}

\theoremstyle{plain}
\newtheorem{theorem}{Theorem}[section]
\newtheorem{proposition}[theorem]{Proposition}
\newtheorem{lemma}[theorem]{Lemma}
\newtheorem{corollary}[theorem]{Corollary}

\theoremstyle{definition}
\newtheorem{definition}[theorem]{Definition}
\newtheorem{example}[theorem]{Example}

\def\w{\widetilde}

\newcommand{\C}{\bm{\mathsf{C}}}
\newcommand{\Cat}{\bm{\mathsf{Cat}}}
\newcommand{\XM}{\bm{\mathsf{XMod}}}
\newcommand{\om}{\omega}
\newcommand{\Om}{\Omega}
\newcommand{\D}{\operatorname{D}}
\newcommand{\Der}{\operatorname{Der}}
\newcommand{\Aut}{\operatorname{Aut}}

\begin{document}
\title{Derived crossed modules}

\author{Tunçar ŞAHAN\thanks{T. Şahan (e-mail : tuncarsahan@gmail.com)}}
\affil{\small{Department of Mathematics, Aksaray University, Aksaray, Turkey}}

\date{}

\maketitle

\begin{abstract}
In this study, we interpret the notion of homotopy of morphisms in the category of crossed modules in a category $\C$ of groups with operations using the categorical equivalence between crossed modules and internal categories in $\C$. Further, we characterize the derivations of crossed modules in a category $\C$ of groups with operations and obtain new crossed modules using regular derivations of old one. 
\end{abstract}

\small{\noindent{\bf Key Words:} Homotopy, crossed module, internal category, group with operations  \\ 
{\bf Classification:} 55U35, 18D35, 20L05, 18A23 }

\section{Introduction}

Crossed modules first appear in Whitehead's work on the second homotopy groups \cite{W46,W49}. During these studies, he observed that the boundary map from the relative second homotopy group to the first homotopy group provided some conditions. He called the algebraic construction that satisfies these conditions a crossed module. Briefly, a crossed module consist of a group homomorphism $\alpha :A\rightarrow B$ and an (left) action of $B$ on $A$ (denoted by $b\cdot a$ for all $a\in A$ and $b\in B$) which satisfies \textbf{(i)} $\alpha \left( b\cdot a \right)=b+\alpha \left( a \right)-b$ and \textbf{(ii)} $\alpha \left( a \right)\cdot {{a}_{1}}=a+{{a}_{1}}-a$ for all $a,{{a}_{1}}\in A$ and $b\in B$. 

Crossed modules are algebraic models for (connected) homotopy 2-types while groups are algebraic models for homotopy 1-types. Crossed modules have been widely used in various areas of mathematics such as homotopy theory \cite{Br99}, representation theory \cite{BrHu82}, algebraic K-theory \cite{Lo78}, and homological algebra \cite{Hu80,Lue81}. 

Brown and Spencer proved that the category of crossed modules over groups and the category of group-groupoids (internal categories with in the category of groups) are equivalent \cite{BrSp76}. Topological version of this equivalence was introduced in \cite{BaSt09}. As a generalisation of the given equivalence in \cite{BrSp76}, Patchkoria \cite{Pat98} defined crossed semimodules as crossed modules in the category of monoids and prove that these objects are categorically equivalent to the Schreier internal categories in the category of monoids. Topological version of this equivalence was given in \cite{Tem16}.

Orzech \cite{Or72} defined the category of groups with operations including groups, $R$-modules, rings without identity, etc. Porter \cite{Por87} proved an analogous result to one given in \cite{BrSp76} for any category $\C$ of groups with operations. Using this equivalence, in \cite{Ak13} coverings of internal groupoids and of crossed modules in $\C$ were introduced and it was proved that these notions are also categorically equivalent. Recently, these results were given for an arbitrary category of topological groups with operations in \cite{MuSa14}.

Brown and Spencer \cite{BrSp76}, also defined the homotopy notion for group-groupoid morphisms using the homotopy of crossed module morphisms defined by Cockcroft \cite{Co52}.

In this paper we define the notion of homotopy of morphisms in the category of crossed modules and of internal categories in a fixed category $\C$ of groups with operations. Further, since a derivation is a special homotopy of morphisms of crossed modules, we characterized the derivations of crossed modules in a category $\C$ of groups with operations. Finally, by using the derivations of a given crossed module we construct new crossed modules which will be called derived crossed modules and obtain an infinite series of isomorphic crossed modules.

\section{Crossed modules and internal categories in C}

Category of groups with operations were introduced by Orzech in \cite{Or72} and arranged in \cite{Por87}. We now recall that the definition of category of groups with operations from \cite{Por87}. From now on $\C$ will be a category of groups with a set of operations $\Omega $ and with a set $E$ of identities such that $E$ includes the group laws, and the following conditions hold: If ${{\Omega }_{i}}$ is the set of $i$-ary operations in $\Omega $, then
\begin{enumerate}
	\item [\textbf{(a)}] $\Omega ={{\Omega }_{0}}\cup {{\Omega }_{1}}\cup {{\Omega }_{2}};$
	\item[\textbf{(b)}] The group operations written additively 0; $-$ and $+$ are the elements of ${{\Omega }_{0}}$, ${{\Omega }_{1}}$ and ${{\Omega }_{2}}$ respectively. 
\end{enumerate}	
Let $\Omega_2'={{\Omega }_{2}}-\left\{ + \right\}$, $\Omega_1'={{\Omega }_{1}}-\left\{ - \right\}$ and assume that if $*\in \Omega_2'$, then ${{*}^{\circ }}\in \Omega_2'$ defined by $a{{*}^{\circ }}b=b*a$ is also in $\Omega_2'$. Also assume that ${{\Omega }_{0}}=\left\{ 0 \right\}$;
\begin{enumerate}	
	\item [\textbf{(c)}] For each $*\in \Omega_2'$, $E$ includes the identity $a*\left( b+c \right)=a*b+a*c$; 
	\item [\textbf{(d)}] For each $\omega \in \Omega_1'$ and  $*\in \Omega_2'$, $E$ includes the identities $\omega \left( a+b \right)=\omega \left( a \right)+\omega \left( b \right)$ and $\omega \left( a \right)*b=\omega \left( a*b \right)$.
\end{enumerate}

A category satisfying the conditions \textbf{(a)-(d)} is called a category of groups with operations \cite{Por87}. 

\begin{example}
	The categories of groups, rings generally without identity, $R$-modules, associative, associative commutative, Lie, Leibniz, alternative algebras are examples of categories of groups with operations.
\end{example}

A morphism in $\C$ is a group homomorphism, which preserves the operations in ${{\Omega }_{1}}$ and ${{\Omega }_{2}}$. 

We recall from \cite{Por87} that for groups with operations $A$ and $B$  an {\em extension} of $A$ by $B$ is
an exact sequence
\[\xymatrix{
	{\mathsf{0}} \ar[r] &   A \ar@{->}[r]^{\imath} &   E \ar@{->}[r]^{p} &   B 
	\ar[r] & {\mathsf{0}} }\]
in which $p$ is surjective and $\imath$ is the kernel of $p$.  An extension of $B$ by $A$ is called {\em split } if there is a morphism $s\colon  B \to E$ such that $p s = 1_B$, i.e. $s$ is a section of $p$.  A split extension of $B$ by $A$ is called a {\em  $B$-structure on $A$}.  For given  such a $B$-structure on $A$ we obtain a set of actions of $B$ on $A$   called {\em derived actions} by Orzech \cite[p.293]{Or72} for $b\in B$, $a\in A$ and $*\in \Omega'_2$
\begin{equation} \label{eq12} \begin{array}{rcl}
		b\cdot a & = &s(b)+a-s(b),\\
		b* a  & = &s(b)* a.
\end{array}\end{equation}

Orzech \cite{Or72} proved that a set of actions of an object $B$ on an object $A$ is a set of derived actions if and only if the semi-direct product $A\rtimes B$ defined by the operations
\begin{align*}
	\left( a,b \right)+\left( {{a}_{1}},{{b}_{1}} \right) &=\left( a+b\cdot {{a}_{1}},b+{{b}_{1}} \right), \\ 
	\left( a,b \right)*\left( {{a}_{1}},{{b}_{1}} \right) &=\left( a*{{a}_{1}}+a*{{b}_{1}}+b*{{a}_{1}},b*{{b}_{1}} \right)\,\,\,\,\text{and} \\ 
	\omega \left( a,b \right) &=\left( \omega \left( a \right),\omega \left( b \right) \right)
\end{align*}
is again an object in $\C$. 

Together with the description of the set of derived actions given above, Datuashvili \cite{Da95} proved the following proposition which gives the necessary and sufficient conditions for a set of actions to be a set of derived actions.

\begin{proposition}\cite{Da95}
	A set of actions in $\C$ is a set of derived actions if and	only if  it satisfies the following conditions:
	\begin{enumerate}[label=\textbf{(\arabic{*})}]
		\item $0\cdot a=a$,
		\item $b\cdot(a_1+a_2)=b\cdot a_1+b\cdot a_2$,
		\item $(b_1+b_2)\cdot a=b_1\cdot(b_2\cdot a)$,
		\item $b*(a_1+a_2)=b* a_1+b* a_2$,
		\item $(b_1+b_2)* a=b_1*a+b_2* a$,
		\item $(b_1*b_2)\cdot(a_1*a_2)=a_1*a_2$,
		\item $(b_1*b_2)\cdot(a*b)=a*b$,
		\item $a_1*(b\cdot a_2)=a_1*a_2$,
		\item $b*(b_1\cdot a)=b*a$,
		\item $\om(b\cdot a)=\om(b)\cdot\om(a)$,
		\item $\om(a*b)=\om(a)*b=a*\om(b)$,
		\item $x*y+z*t=z*t+x*y$,
	\end{enumerate}	
	for each $\om\in\Om_1'$, $*\in\Om_2{}'$, $b$, $b_1$, $b_2\in B$, $a,a_1,a_2\in A$ and for $x,y,z,t\in A\cup B$ whenever each side of \textbf{(12)} has a sense.
\end{proposition}

Porter \cite{Por87} defines a crossed module in $\C$ as in the following way.

\begin{definition}\cite{Por87}
	A crossed module in $\C$ is a morphism $\alpha\colon A\rightarrow B$ in $\C$ with a set of derived actions of $B$ on $A$ such that $(1_A,\alpha)\colon A\rtimes A\rightarrow A\rtimes B$ and $(\alpha,1_B)\colon A\rtimes B\rightarrow B\rtimes B$ are morphisms in $\C$ where $A\rtimes A$ and $B\rtimes B$ are semi-direct products constructed by conjugation actions on $A$ and on $B$ respectively.
\end{definition}

Porter \cite{Por87} also gives the necessary and sufficient conditions for a morphism in $\C$ to be a crossed module with a given set of derived actions in terms of operations and actions.

\begin{proposition}\cite{Por87}\label{xmodprop}
	A crossed module $\left( A,B,\alpha  \right)$ in $\C$ consist of a morphism $\alpha $ in $\C$ and a set of derived actions of $B$ on $A$ with the conditions: 
	\begin{enumerate}[label=\textbf{(CM\arabic{*})}, leftmargin=1.7cm]
		\item $\alpha(b \cdot a) = b + \alpha(a) - b$;
		\item $\alpha(a)\cdot a_1=a+a_1-a$;
		\item $\alpha(b\ast a)=b\ast \alpha(a)$, $\alpha(a*b)=\alpha(a)*b$;
		\item $\alpha(a)\ast a_1=a\ast a_1=a*\alpha(a_1)$
	\end{enumerate}
	for all $a,{{a}_{1}}\in A$, $b\in B$ and $*\in \Omega_2'$.
\end{proposition} 

\begin{definition}\label{xmodmorp}
	A morphism $\left( {{f}_{1}},{{f}_{0}} \right)$ between crossed modules $\left( A,B,\alpha  \right)$ and $\left( A',B',\alpha ' \right)$ is a pair ${{f}_{1}}:A\rightarrow A'$ and ${{f}_{0}}:B\rightarrow B'$ of morphisms in $\C$ such that
	\begin{enumerate}[label=\textbf{(\roman{*})}, leftmargin=1.5cm]
		\item ${{f}_{0}}\alpha=\alpha '{{f}_{1}}$,
		\item ${{f}_{1}}\left( b\cdot a \right)\text{ }={{f}_{0}}\left( b \right)\cdot {{f}_{1}}\left( a \right)$ and
		\item ${{f}_{1}}\left( b*a \right)\text{ }={{f}_{0}}\left( b \right)*{{f}_{1}}\left( a \right)$.
	\end{enumerate}
	for all $a\in A$, $b\in B$ and $*\in \Omega_2'$.
\end{definition}

Crossed modules and crossed module morphisms form a category which is denoted by $\XM$($\C$).

\begin{definition}\cite{DatThs}
	An internal category $C$ in $\C$ consist of two objects ${{C}_{1}}$ (object of morphisms) and ${{C}_{0}}$ (object of objects) in $\C$, initial and final point maps ${{d}_{0}},{{d}_{1}}:{{C}_{1}}\rightarrow {{C}_{0}}$, identity morphism map $\varepsilon :{{C}_{0}}\rightarrow {{C}_{1}}$ and partial composition \[\circ :{{C}_{1}}{{\,}_{{{d}_{0}}}}{{\times }_{{{d}_{1}}}}\,{{C}_{1}}\rightarrow {{C}_{1}},\left( b,a \right)\mapsto b\circ a\] which are morphisms in $\C$.
\end{definition}

It is easy to see that partial composition being a morphism in $\C$, implies that 
\[\left( a*b \right)\circ \left( c*d \right)=\left( a\circ c \right)*\left( b\circ d \right)\]
for all $a,b,c,d\in {{C}_{1}}$ and $*\in {{\Omega }_{2}}$, whenever one side makes sense. Equality given above is called interchange law. Partial composition can be described in terms of addition operation (+) using interchange law as follows:
\begin{align*}
	b\circ a & =\left( b+{{1}_{0}} \right)\circ \left( {{1}_{y}}+\left( -{{1}_{y}}+a \right) \right) \\ 
	&=\left( b\circ {{1}_{y}} \right)+\left( {{1}_{0}}\circ \left( -{{1}_{y}}+a \right) \right) \\ 
	&=b-{{1}_{y}}+a 
\end{align*}
for $a,b\in {{C}_{1}}$ such that ${{d}_{1}}\left( a \right)=y={{d}_{0}}\left( b \right)$. Another application of interchange law is that any internal category in $\C$ is an internal groupoid since
\[{{a}^{-1}}={{1}_{{{d}_{1}}\left( a \right)}}-a+{{1}_{{{d}_{0}}\left( a \right)}}\]
is the inverse morphism of $a\in {{C}_{1}}$. From now on, an internal category (groupoid) will be denoted by $G$. 

Morphisms between internal groupoids in $\C$ are groupoid morphisms (functors) that are morphisms in $\C$. Internal groupoids form a category with internal groupoid morphisms. This category is denoted by $\Cat(\C)$. 

\begin{example}
	If $A$ is an object in $\C$ then $G=A\times A$ is an internal groupoid in $\C$.
\end{example} 

Porter \cite{Por87} proved a similar result to Brown \& Spencer Theorem \cite[Theorem 1]{BrSp76} for an arbitrary category $\C$ of groups with operations. We sketch the proof of Porter's Theorem since we need some details later.

\begin{theorem}\cite{Por87}\label{catxmod}
	The category $\XM$($\C$) of crossed modules in $\C$ and the category $\Cat$($\C$) of internal categories (groupoids) in $\C$ are equivalent.
\end{theorem}

\begin{proof}
	Define a functor $\delta $ from $\XM$($\C$) to $\Cat$($\C$) as follows: Let $\left( A,B,\alpha  \right)$ be an object in $\XM$($\C$). Then the internal groupoid $\delta \left( A,B,\alpha  \right)$ has object of morphisms $A\rtimes B$, object of objects $B$ where ${{d}_{0}}\left( a,b \right)=b$, ${{d}_{1}}\left( a,b \right)=\alpha \left( a \right)+b$, $\varepsilon \left( b \right)=\left( 0,b \right)$ and the partial composition
	\[\left( a,b \right)\circ \left( {{a}_{1}},{{b}_{1}} \right)=\left( a+{{a}_{1}},b \right).\]
	
	Conversely, define a functor $\theta $ from $\Cat$($\C$) to $\XM$($\C$) as follows: Let G be an internal groupoid in $\C$. Then $\theta \left( G \right)=\left( \ker {{d}_{0}},{{G}_{0}},{{d}_{1}}_{|\ker {{d}_{0}}} \right)$ where the action of ${{G}_{0}}$ on $\ker {{d}_{0}}$ is defined by
	\begin{align*}
		x\cdot a &={{1}_{x}}+a-{{1}_{x}} \\ 
		x*a &={{1}_{x}}*a
	\end{align*}
	for all $*\in \Omega_2'$. Other details are straightforward.	
\end{proof}

\section{Homotopy of crossed module morphisms in C}

In this section, first we define the homotopy $\eta :f\Rightarrow g$ of internal groupoid morphisms $f$ and $g$ in $\C$. Using the results of Cockcroft \cite{Co52} and Brown and Spencer \cite{BrSp76} we can give the following definition.

\begin{definition}
	Let $G$ and $H$ be two internal groupoids in $\C$ and $f,g:G\rightarrow H$ be two internal groupoid morphisms. Then we call that $f$ is homotopic to $g$ if there is a natural isomorphism $\eta$ from $f$ to $g$ and 
	$\eta :{{G}_{0}}\rightarrow {{H}_{1}}$ is a morphism in $\C$. We denote this by $\eta :f\simeq g$.
\end{definition}

$\Cat$($\C$) has, with its notion of homotopy, the structure of a 2-category, where 2-cells are homotopies.

Now, let $\left( A,B,\alpha  \right)$ and $\left( C,D,\gamma  \right)$ be two crossed modules and by Theorem \ref{catxmod}, $G$ and $H$ be the corresponding internal groupoids to $\left( A,B,\alpha  \right)$ and $\left( C,D,\gamma  \right)$ respectively. Also, let 
\[\left( {{f}_{1}},{{f}_{0}} \right),\left( {{g}_{1}},{{g}_{0}} \right):\left( A,B,\alpha  \right)\rightarrow \left( C,D,\gamma  \right)\]
be two crossed module morphisms. In this case the corresponding internal groupoid morphisms to $\left( {{f}_{1}},{{f}_{0}} \right)$ is ${{f}_{1}}\times {{f}_{0}}$ on morphisms and ${{f}_{0}}$ on objects and to $\left( {{g}_{1}},{{g}_{0}} \right)$ is ${{g}_{1}}\times {{g}_{0}}$ on morphisms and ${{g}_{0}}$ on objects. Assume that we have a function $d$ from $B$ to $C$.
\[d:B\rightarrow C\]
Hence we obtain a function 
\[\begin{array}{*{35}{l}}
V & : & B & \rightarrow  & C\rtimes D  \\
{} & {} & b & \mapsto  & V\left( b \right)=\left( d\left( b \right),{{f}_{0}}\left( b \right) \right).  \\
\end{array}\]
Suppose that this function is a homotopy between $\left( {{f}_{1}}\times {{f}_{0}},{{f}_{0}} \right)$ and $\left( {{g}_{1}}\times {{g}_{0}},{{g}_{0}} \right)$. Thus $V$ must be a natural isomorphism and a morphism in $\C$. Now let us find the conditions on $d$ to be a homotopy.

Since $V$ is a morphism in $\C$ then 
\[V\left( b+{{b}_{1}} \right)=V\left( b \right)+V\left( {{b}_{1}} \right),\]
\[V\left( b*{{b}_{1}} \right)=V\left( b \right)*V\left( {{b}_{1}} \right)\]
and
\[V\left( \omega \left( b \right) \right)=\omega \left( V\left( b \right) \right)\]
for all $b,{{b}_{1}}\in B$, $*\in \Omega_2'$ and $\omega \in \Omega_1$. Then 
\begin{align*}
	\left( d\left( b+{{b}_{1}} \right),{{f}_{0}}\left( b+{{b}_{1}} \right) \right) &=V\left( b+{{b}_{1}} \right) \\ 
	&=V\left( b \right)+V\left( {{b}_{1}} \right) \\ 
	&=\left( d\left( b \right),{{f}_{0}}\left( b \right) \right)+\left( d\left( {{b}_{1}} \right),{{f}_{0}}\left( {{b}_{1}} \right) \right) \\ 
	&=\left( d\left( b \right)+{{f}_{0}}\left( b \right)\cdot d\left( {{b}_{1}} \right),{{f}_{0}}\left( b \right)+{{f}_{0}}\left( {{b}_{1}} \right) \right)
\end{align*}
and thus
\begin{equation}\label{3-1}
	d\left( b+{{b}_{1}} \right)=d\left( b \right)+{{f}_{0}}\left( b \right)\cdot d\left( {{b}_{1}} \right).
\end{equation}
Also
\begin{align*}
	\left( d\left( b*{{b}_{1}} \right),{{f}_{0}}\left( b*{{b}_{1}} \right) \right)&=V\left( b*{{b}_{1}} \right) \\ 
	&=V\left( b \right)*V\left( {{b}_{1}} \right) \\ 
	&=\left( d\left( b \right),{{f}_{0}}\left( b \right) \right)*\left( d\left( {{b}_{1}} \right),{{f}_{0}}\left( {{b}_{1}} \right) \right) \\ 
	&=\left( d\left( b \right)*d\left( {{b}_{1}} \right)+d\left( b \right)*{{f}_{0}}\left( {{b}_{1}} \right)+{{f}_{0}}\left( b \right)*d\left( {{b}_{1}} \right),{{f}_{0}}\left( b \right)*{{f}_{0}}\left( {{b}_{1}} \right) \right)
\end{align*}
and thus
\begin{equation}\label{3-2}
	d\left( b*{{b}_{1}} \right)=d\left( b \right)*d\left( {{b}_{1}} \right)+d\left( b \right)*{{f}_{0}}\left( {{b}_{1}} \right)+{{f}_{0}}\left( b \right)*d\left( {{b}_{1}} \right).
\end{equation}
Finally,
\begin{align*}
	\left( d\left( \omega \left( b \right) \right),{{f}_{0}}\left( \omega \left( b \right) \right) \right)&=V\left( \omega \left( b \right) \right) \\ 
	&=\omega \left( V\left( b \right) \right) \\ 
	&=\left( \omega \left( d\left( b \right) \right),\omega \left( {{f}_{0}}\left( b \right) \right) \right)
\end{align*}
and thus 
\begin{equation}\label{3-3}
	d\left( \omega \left( b \right) \right)=\omega \left( d\left( b \right) \right).
\end{equation}

On the other hand, since $V$ is a natural isomorphism then first of all, the morphism $V\left( b \right)=\left( d\left( b \right),{{f}_{0}}\left( b \right) \right)$ in ${{H}_{1}}$ has ${{f}_{0}}\left( b \right)$ as initial point object and $\gamma \left( d\left( b \right) \right)+{{f}_{0}}\left( b \right)={{g}_{0}}\left( b \right)$ as final point object according to proof of Theorem \ref{catxmod}. Thus
\begin{equation}\label{3-4}
	\gamma \left( d\left( b \right) \right)={{g}_{0}}\left( b \right)-{{f}_{0}}\left( b \right).
\end{equation}
Aslo for a morphism $\left( a,b \right):b\rightarrow \alpha \left( a \right)+b$ in ${{G}_{1}}=A\rtimes B$, 
\[\left( {{g}_{1}}\left( a \right),{{g}_{0}}\left( b \right) \right)\circ V\left( b \right)=V\left( \alpha \left( a \right)+b \right)\circ \left( {{f}_{1}}\left( a \right),{{f}_{0}}\left( b \right) \right).\]
After some calculations one can easily see that the equation above implies
\begin{equation}\label{3-5}
	d\left( \alpha \left( a \right) \right)={{g}_{1}}\left( a \right)-{{f}_{1}}\left( a \right).
\end{equation}

According to equations (\ref{3-1})-(\ref{3-5}) now we can give the definition of homotopy of crossed module morphisms in $\C$ as follows.

\begin{definition}
	Let $\left( A,B,\alpha  \right)$ and $\left( C,D,\gamma  \right)$ be two crossed modules and
	\[\left( {{f}_{1}},{{f}_{0}} \right),\left( {{g}_{1}},{{g}_{0}} \right):\left( A,B,\alpha  \right)\rightarrow \left( C,D,\gamma  \right)\]	
	be two crossed module morphisms. If there is a function $d:B\rightarrow C$ satisfying
	\begin{enumerate}[label=\textbf{(\roman{*})}, leftmargin=1.5cm]
		\item $d\left( b+{{b}_{1}} \right)=d\left( b \right)+{{f}_{0}}\left( b \right)\cdot d\left( {{b}_{1}} \right)$,
		\item $d\left( b*{{b}_{1}} \right)=d\left( b \right)*d\left( {{b}_{1}} \right)+d\left( b \right)*{{f}_{0}}\left( {{b}_{1}} \right)+{{f}_{0}}\left( b \right)*d\left( {{b}_{1}} \right)$,
		\item $d\left( \omega \left( b \right) \right)=\omega \left( d\left( b \right) \right)$,
		\item $\gamma \left( d\left( b \right) \right)={{g}_{0}}\left( b \right)-{{f}_{0}}\left( b \right)$ and
		\item $d\left( \alpha \left( a \right) \right)={{g}_{1}}\left( a \right)-{{f}_{1}}\left( a \right)$
	\end{enumerate}
	for all $a\in A$, $b,{{b}_{1}}\in B$, $*\in \Omega_2'$ and $\omega \in \Omega_1'$ then we say that $\left( {{f}_{1}},{{f}_{0}} \right)$ and $\left( {{g}_{1}},{{g}_{0}} \right)$ are homotopic and that $d$ is a homotopy from $\left( {{f}_{1}},{{f}_{0}} \right)$ to $\left( {{g}_{1}},{{g}_{0}} \right)$. We denote this by $d:\left( {{f}_{1}},{{f}_{0}} \right)\Rightarrow \left( {{g}_{1}},{{g}_{0}} \right)$ or by $d:\left( {{f}_{1}},{{f}_{0}} \right)\simeq \left( {{g}_{1}},{{g}_{0}} \right)$.
\end{definition}

This notion of homotopy gives $\XM$($\C$) the structure of a 2-category, where 2-cells arre homotopies.

\begin{theorem}
	The 2-category $\XM$($\C$) of crossed modules in $\C$ where 2-cells are homotopies and the 2-category $\Cat$($\C$) of internal categories in $\C$ where 2-cells are homotopies are equivalent 2-categories.
\end{theorem}
\begin{proof}
	It is straightforward from the construction of homotopy notion in $\XM$($\C$).
\end{proof}

\section{Derived crossed modules}

Derivations of crossed modules have been defined by Whitehead \cite{W49} according to define the actor crossed module analogous to the automorphism group. Derivations, in fact, are homotopies between crossed module endomorphisms and identity crossed module morphism. 

\begin{definition}\label{der}
	Let $\left( A,B,\alpha  \right)$ be a crossed module. A function $d:B\rightarrow A$ is called a \textit{derivation} if the following conditions hold:
	\begin{enumerate}[label=\textbf{(\roman{*})}, leftmargin=1.5cm]
		\item $d\left( b+{{b}_{1}} \right)=d\left( b \right)+b\cdot d\left( {{b}_{1}} \right)$,
		\item $d\left( b*{{b}_{1}} \right)=d\left( b \right)*d\left( {{b}_{1}} \right)+d\left( b \right)*{{b}_{1}}+b*d\left( {{b}_{1}} \right)$,
		\item $d\left( \omega \left( b \right) \right)=\omega \left( d\left( b \right) \right)$
	\end{enumerate}
	for all $b,{{b}_{1}}\in B$, $*\in \Omega_2'$ and $\omega \in \Omega_1'$.
\end{definition}

The set of all derivations from $B$ to $A$ is denoted by $\Der_{\alpha}\left( B,A \right)$. 

\begin{proposition}
	Let $( A,B,\alpha)$ be a crossed module in $\C$ and $d\in \Der_{\alpha}\left( B,A \right)$. Then \[(\theta_d,\sigma_d)\colon (A,B,\alpha)\rightarrow (A,B,\alpha)\] is a morphism in $\XM(\C)$ where ${{\theta }_{d}}\left( a \right)=d\alpha \left( a \right)+a$ for $a\in A$ and ${{\sigma }_{d}}\left( b \right)=\alpha d\left( b \right)+b$ for $b\in B$.
\end{proposition}

\begin{proof}
	First we need to show that $\theta_d$ and $\sigma_d$ are morphisms in $\C$. In \cite{W48} it has been shown that $\theta_d$ and $\sigma_d$ are group homomorphisms. Now, let $a,a_1\in A$. Then for all $*\in\Om_2'$,
	\[\begin{array}{rl}
	\theta_d(a*a_1)  & = d\alpha(a*a_1)+a*a_1 \\
	& = d(\alpha(a)*\alpha(a_1))+a*a_1 \\
	& = d\alpha(a)*d\alpha(a_1)+d\alpha(a)*\alpha(a_1)+\alpha(a)*d\alpha(a_1)+a*a_1\\
	& = d\alpha(a)*d\alpha(a_1)+d\alpha(a)*a_1+a*d\alpha(a_1)+a*a_1 \\
	& = d\alpha(a)*(d\alpha(a_1)+a_1)+a*(d\alpha(a_1)+a_1)\\
	& = (d\alpha(a)+a)*(d\alpha(a_1)+a_1)\\
	& = \theta_d(a)*\theta_d(a_1)
	\end{array}\]
	and for all $\om\in\Om_1'$
	\[\begin{array}{rl}
	\theta_d(\om(a))  & = d\alpha(\om(a))+\om(a) \\
	& = \om(d\alpha(a))+\om(a) \\
	& = \om(d(\alpha(a))+a) \\
	& = \om(\theta_d(a)).
	\end{array}\]
	Similarly, let $b,b_1\in B$. Then for all $*\in\Om_2'$,
	\[\begin{array}{rl}
	\sigma_d(b*b_1)  & = \alpha d(b*b_1)+b*b_1 \\
	& = \alpha(d(b)*d(b_1)+d(b)*b_1+b*d(b_1))+b*b_1 \\
	& = \alpha d(b)*\alpha d(b_1)+\alpha(d(b)*b_1)+\alpha(b*d(b_1))+b*b_1\\
	& = \alpha d(b)*\alpha d(b_1)+\alpha d(b)*b_1+b*\alpha d(b_1)+b*b_1 \\
	& = \alpha d(b)*(\alpha d(b_1)*b_1)+b_1*(\alpha d(b_1)*b_1) \\
	& = (\alpha d(b)*b)*(\alpha d(b_1)*b_1) \\
	& = \sigma_d(b)*\sigma_d(b_1)
	\end{array}\]
	and for all $\om\in\Om_1'$
	\[\begin{array}{rl}
	\sigma_d(\om(b))  & = \alpha d(\om(b))+\om(b) \\
	& = \om(\alpha d(b))+\om(b) \\
	& = \om(\alpha d(b)+b) \\
	& = \om(\sigma_d(b)).
	\end{array}\]
	Hence $\theta_d$ and $\sigma_d$ are morphisms in $\C$. Now we need to show that $(\theta_d,\sigma_d)$ satisfies the condition of Definition \ref{xmodmorp}.
	\begin{enumerate}[label=\textbf{(\roman{*})}, leftmargin=1.5cm]
		\item Let $a\in A$. Then
		\[\begin{array}{rl}
		\sigma_d\alpha (a) & = \alpha d(\alpha(a))+\alpha(a) \\
		& = \alpha (d \alpha(a)+ a) \\
		& =\alpha \theta_d(a).
		\end{array}\]
		\item Let $a\in A$ and $b\in B$. Then
		\[\begin{array}{rl}
		\theta_d(b\cdot a) & = d \alpha(b\cdot a)+ b\cdot a \\
		& = d(b+\alpha(a)-b) + b\cdot a \\
		& = d(b) + b\cdot d(\alpha(a)-b) + b\cdot a \\
		& = d(b) + b\cdot ( d(\alpha(a)-b) + a ) \\
		& = d(b) + b\cdot ( d\alpha(a) + \alpha(a)\cdot d(-b) + a ) \\
		& = d(b) + b\cdot ( d\alpha(a) + a + d(-b) -a + a ) \\
		& = d(b) + b\cdot ( d\alpha(a) + a + d(-b) ) \\
		& = d(b) + b\cdot ( d\alpha(a) + a ) + b\cdot ( d(-b) ) \\
		& = d(b) + b\cdot ( d\alpha(a) + a ) - d(b) \\
		& = \alpha d(b) \cdot ( b\cdot ( d\alpha(a) + a ) ) \\
		& = ( \alpha d(b) + b ) \cdot ( d\alpha(a) + a )  \\
		& = \sigma_d(b)\cdot\theta_d(a).
		\end{array}\]
		\item Let $a\in A$ and $b\in B$. Then
		\[\begin{array}{rl}
		\theta_d(b * a) & = d \alpha(b * a)+ b * a \\
		& = d(b * \alpha(a)) + b * a \\
		& = d(b)*d\alpha(a) + d(b)*\alpha(a) + b*d\alpha(a) + b * a \\
		& = d(b)*d\alpha(a) + d(b)*a + b*d\alpha(a) + b * a \\
		& = d(b)*(d\alpha(a) + a) + b*(d\alpha(a) + a) \\
		& = \alpha d(b)*(d\alpha(a) + a) + b*(d\alpha(a) + a) \\
		& = (\alpha d(b)+b)*(d\alpha(a) + a) \\
		& = \sigma_d(b)*\theta_d(a).
		\end{array}\] 
	\end{enumerate}
	Thus $(\theta_d,\sigma_d)$ is a morphism in $\XM(\C)$.
\end{proof}

That is, for a crossed module $\left( A,B,\alpha  \right)$, each derivation defines a crosed module endomorphism of $\left( A,B,\alpha  \right)$. Here also note that $\theta_d d=d\sigma_d $ for all $d\in\Der_{\alpha}\left( B,A \right)$.

Now we can define a multiplication on $\Der_{\alpha}\left( B,A \right)$ as in \cite{W48} using the horizontal composition of corresponding natural isomorphisms. Let ${{d}_{1}},{{d}_{2}}\in \Der_{\alpha}\left( B,A \right)$. Then
\[d\left( b \right)=\left( {{d}_{1}}\circ {{d}_{2}} \right)\left( b \right)={{d}_{1}}{{\sigma }_{2}}\left( b \right)+{{d}_{2}}\left( b \right)\]
or equally
\[d\left( b \right)=\left( {{d}_{1}}\circ {{d}_{2}} \right)\left( b \right)={{\theta }_{1}}{{d}_{2}}\left( b \right)+{{d}_{1}}\left( b \right).\]

Zero morphism $0\colon B\rightarrow A$ (which assigns every element of $B$ to the zero element of $A$) is a derivation and the identity element of the multiplication defined above. In this case $\theta_0$ and $\sigma_0$ are the identity morphisms on $A$ and $B$ respectively. Moreover, if $d={{d}_{1}}\circ {{d}_{2}}$ then ${{\sigma }_{d}}={{\sigma }_{{{d}_{1}}}}\circ {{\sigma }_{{{d}_{2}}}}$ and ${{\theta }_{d}}={{\theta }_{{{d}_{1}}}}\circ {{\theta }_{{{d}_{2}}}}$. A derivation is called \textit{regular} if it has an inverse up to the composition defined above. Regular derivations forms a group structure which is denoted by $\D_{\alpha}\left( B,A \right)$ and called the \textit{Whitehead group}.

\begin{lemma}
	Let $\left( A,B,\alpha  \right)$ be a crossed module and $d\in \Der_{\alpha}\left( B,A \right)$ such that $d(b)\in\ker\alpha$ for all $b\in B$. Then $d=0$.
\end{lemma} 

Following proposition is a generalisation of the one given in \cite{W48} and of the one combined in \cite{Nor90} to an arbitrary category $\C$ of groups with operations.

\begin{proposition}\label{ucluprop}
	Let $\left( A,B,\alpha  \right)$ be a crossed module in $\C$ and $d\in \Der_{\alpha}\left( B,A \right)$. Then the following are equivalent.
	\begin{enumerate}[label=\textbf{(\alph{*})}, leftmargin=1.5cm]
		\item\label{i} $d\in \D_{\alpha}\left( B,A \right)$
		\item\label{ii} $\theta_d \in \Aut\left( A \right)$
		\item\label{iii} $\sigma_d \in \Aut\left( B \right)$
	\end{enumerate}
\end{proposition}

\begin{proof} Let $\left( A,B,\alpha  \right)$ be a crossed module and $d\in \Der_{\alpha}\left( B, A \right)$. 
	\begin{enumerate}[leftmargin=1.7cm]
		\item[\textbf{(a)$\Rightarrow$(b)}] Assume that $d\in D(B,A)$. Then there exist a derivation $e\in\D_{\alpha}\left( B,A \right) $ such that $d\circ e=0\colon B\rightarrow A$. Thus
		\[\theta_{d\circ e}(a) = \theta_{d}(\theta_{e}(a)) = a \]
		for all $a\in A$. Hence $\theta_{e}$ is the inverse of $\theta_{d}$. So $\theta_d \in \Aut\left( A \right)$.
		\item[\textbf{(b)$\Rightarrow$(a)}] Now assume that $\theta_d \in \Aut\left( A \right)$. Let define a function $e\colon B\rightarrow A$ with $e(b)=\theta_{d}^{-1}(-d(b))$. First of all, it easy to see that $\theta_d(e(b))=-d(b)$ for all $b\in B$. Now we show that $e\in \Der_{\alpha}(B,A)$.		
		\begin{enumerate}[label=\textbf{(\roman{*})}, leftmargin=-0.5cm]
			\item Let $b,{{b}_{1}}\in B$. Then
			\[\begin{array}{rl}
			\theta_d\left( e(b)+b\cdot e(b_1)\right)  & = \theta_d(e(b))+\theta_d(b\cdot e(b_1))\\
			& = \theta_d(e(b))+\sigma_d(b)\cdot\theta_d(e(b_1))\\
			& = -d(b)+(\alpha(d(b))+b)\cdot(-d(b_1))\\
			& = -d(b)+\alpha(d(b))\cdot (b\cdot(-d(b_1)))\\
			& = -d(b)+d(b)+(b\cdot(-d(b_1)))-d(b)\\
			& = -(d(b)+b\cdot(d(b_1)))\\
			& = -(d(b+b_1))\\
			& = \theta_d(e(b+b_1)).
			\end{array}\]
			However, since $\theta_{d}\in \Aut(A)$ then $e(b+b_1)=e(b)+b\cdot e(b_1)$.
			\item Let $b,{{b}_{1}}\in B$ and $\ast\in\Omega_{2}'$. If we set $x:=e(b)*e(b_1)+e(b)*b_1+b*e(b_1)$ then
			\[\begin{array}{rl}
			\theta_d\left( x\right)  & = \theta_d( e(b)*e(b_1))+\theta_d(e(b)*b_1)+\theta_d(b*e(b_1))\\
			& = d(b)*d(b_1)+(-d(b))*\sigma_d(b_1)+\sigma_d(b)*(-d(b_1))\\
			& = d(b)*d(b_1)+(-d(b))*(\alpha(d(b_1)+b_1))+\sigma_d(b)*(-d(b_1))\\
			& = d(b)*d(b_1)+(-d(b))*d(b_1)+((-d(b))*b_1))+\sigma_d(b)*(-d(b_1))\\
			& = d(b)*d(b_1)+-(d(b)*d(b_1))+((-d(b))*b_1))+\sigma_d(b)*(-d(b_1))\\
			& = ((-d(b))*b_1))+\sigma_d(b)*(-d(b_1))\\
			& = ((-d(b))*b_1))+(\alpha(d(b)+b))*(-d(b_1))\\
			& = ((-d(b))*b_1))+d(b)*(-d(b_1))+b*(-d(b_1))\\
			& = -(d(b)*d(b_1)+d(b)*b_1+b*d(b_1))\\
			& = -(d(b*b_1))\\
			& = \theta_d(e(b*b_1)).
			\end{array}\]
			Again, since $\theta_{d}\in \Aut(A)$ then $e(b*b_1)=e(b)*e(b_1)+e(b)*b_1+b*e(b_1)$.
			\item Finally, let $b\in B$ and $\om\in\Omega_{1}'$. Then
			\[\begin{array}{rl}
			e(\om (b))  & = \theta_{d}^{-1}(-d(\om(b)))\\
			& = \theta_{d}^{-1}(\om(-d(b))) \\
			& = \om(\theta_{d}^{-1}(-d(b))) \\
			& = \om(e(b)).
			\end{array}\]
		\end{enumerate}
		Thus $e\colon B\rightarrow A$ is in $\Der_{\alpha}(B,A)$ by \textbf{(i)-(iii)}. Now we show that $e$ is the inverse of $d$. Let $b\in B$. Then
		\[\begin{array}{rl}
		(d\circ e)(b)  & = \theta_{d}(e(b))+d(b)\\
		& = \theta_{d}(\theta_{d}^{-1}(-d(b)))+d(b) \\
		& = d(b)+(-d(b)) \\
		& = 0
		\end{array}\]
		and 
		\[\begin{array}{rl}
		(e\circ d)(b)  & = \theta_{e}(d(b))+e(b)\\
		& = e(\alpha(d(b)))+d(b)+\theta_{d}^{-1}(-d(b)) \\
		& = \theta_{d}^{-1}(-d(\alpha(d(b))))+d(b)+\theta_{d}^{-1}(-d(b)) \\
		& = -\theta_{d}^{-1}(d(\alpha(d(b))))+d(b)+\theta_{d}^{-1}(-d(b))\\
		& = -\theta_{d}^{-1}(\theta_{d}(d(b))-d(b))+d(b)+\theta_{d}^{-1}(-d(b)) \\
		& = \theta_{d}^{-1}(d(b))+\theta_{d}^{-1}(-d(b))\\
		& = \theta_{d}^{-1}(d(b)+(-d(b)))\\
		& = \theta_{d}^{-1}(0)\\
		& = 0.
		\end{array}\]
		Hence $e$ is the inverse of $d$. So $d\in \D_{\alpha}\left( B,A \right)$.
	\end{enumerate}
	
	Similarly one can show that \textbf{(a)$\Leftrightarrow$(c)}. Here, in the proof of \textbf{(c)$\Rightarrow$(a)}, the inverse $e\colon B\rightarrow A$ of a derivation $d\in\Der_{\alpha}(B,A)$ is given by $e(b)=-d(\sigma_{d}^{-1}(b))$ for all $b\in B$. This completes the proof.
\end{proof}

Now, from a crossed module in $\C$, we obtain new crossed modules on the same objects using the derivations of old one.

\begin{proposition}\label{newact1}
	Let $(A,B,\alpha)$ be a crossed module in $\C$ and $d\in \Der_{\alpha}(B,A)$. Then the set of actions defined by
	\[\begin{array}{rcl}
	{^b}a & = & d(b)+b\cdot a-d(b)\\
	b ~\tilde{*}~ a & = & d(b)*a+b*a 
	\end{array}\]
	is a set of derived actions.
\end{proposition}

\begin{proof}
	Since $(A,B,\alpha)$ is a crossed module then we have a split extension
	\[\xymatrix{
		{\mathsf{0}} \ar[r] &   A \ar@{->}[r]^-{\imath} &   A\rtimes B \ar@{->}[r]_-{p} &   B \ar@/_/[l]_-{s}
		\ar[r] & {\mathsf{0}} }\]
	of $B$ by $A$ where $\imath(a)=(a,0)$, $p(a,b)=b$ and $s(b)=(0,b)$. Also we know that 
	\[\begin{array}{rcccl}
	V & : & B & \rightarrow  & A\rtimes B  \\
	{} & {} & b & \mapsto  & V\left( b \right)=\left( d\left( b \right),b \right) \\
	\end{array}\]
	is a morphism in $\C$. Then we obtain a new split extension
	\[\xymatrix{
		{\mathsf{0}} \ar[r] &   A \ar@{->}[r]^-{\imath} &   A\rtimes B \ar@{->}[r]_-{p} &   B \ar@/_/[l]_-{V}
		\ar[r] & {\mathsf{0}} }\]
	of $B$ by $A$. Derived actions obtained from this split extensions are given by
	\[\begin{array}{rl}
	\left( {^b}a , 0 \right)   & = V(b)+(a,0)-V(b)\\
	& = (d(b),b)+(a,0)-(d(b),b)\\
	& = (d(b)+b\cdot a,b)+((-b)\cdot(-d(b)),-b) \\
	& = (d(b)+b\cdot a-d(b),0)
	\end{array}\]
	and 
	\[\begin{array}{rl}
	\left( b ~\tilde{*}~ a , 0 \right)   & = V(b)*(a,0)\\
	& = (d(b),b)*(a,0)\\
	& = (d(b)*a+d(b)*0+b*a,b*0) \\
	& = (d(b)*a+b*a,0)
	\end{array}\]
	for all $a\in A$, $b\in B$ and $*\in\Omega_{2}'$. Thus the set of actions given by
	\[\begin{array}{rcl}
	{^b}a & = & d(b)+b\cdot a-d(b)\\
	b ~\tilde{*}~ a & = & d(b)*a+b*a 
	\end{array}\]
	is a set of derived actions.
\end{proof}

So, every derivation $d\in \Der_{\alpha}(B,A)$ defines a split extension of $B$ by $A$ hence a new set of derived actions of $B$ on $A$. Now we obtain this new set of derived actions using regular derivations.

\begin{proposition}\label{newact}
	Let $(A,B,\alpha)$ be a crossed module in $\C$ and $d\in \D_{\alpha}(B,A)$. Then the set of actions defined by
	\[\begin{array}{rcl}
	{^b}a & = & \sigma_{d}(b)\cdot a\\
	b ~\tilde{*}~ a & = & \sigma_{d}(b)*a 
	\end{array}\]
	is a set of derived actions.
\end{proposition}

\begin{proof}
	Since $(A,B,\alpha)$ is a crossed module then we have a split extension
	\[\xymatrix{
		{\mathsf{0}} \ar[r] &   A \ar@{->}[r]^-{\imath} &   A\rtimes B \ar@{->}[r]_-{p} &   B \ar@/_/[l]_-{s}
		\ar[r] & {\mathsf{0}} }\]
	of $B$ by $A$ where $\imath(a)=(a,0)$, $p(a,b)=b$ and $s(b)=(0,b)$. Also, since $d\in \D_{\alpha}(B,A)$ then $\sigma_d\colon B\rightarrow B$ is an isomorphism. Hence we obtain a new split extension
	\[\xymatrix{
		{\mathsf{0}} \ar[r] &   A \ar@{->}[r]^-{\imath} &   A\rtimes B \ar@{->}[r]_-{p_{d}} &   B \ar@/_/[l]_-{s_{d}}
		\ar[r] & {\mathsf{0}} }\]
	of $B$ by $A$ where $s_d=s\sigma_d$ and $p_d=\sigma_{d}^{-1}p$. Derived actions which are obtained from this split extension are given by
	\[\begin{array}{rcl}
	{^b}a & = & \sigma_{d}(b)\cdot a\\
	b ~\tilde{*}~ a & = & \sigma_{d}(b)*a 
	\end{array}\]
	for all $a\in A$, $b\in B$ and $*\in\Omega_{2}'$. This completes the proof.
\end{proof}

\begin{corollary}
	Let $(A,B,\alpha)$ be a crossed module in $\C$. Then for the zero derivation $0\in \Der_{\alpha}(B,A)$ 
	\[\begin{array}{rcl}
	{^b}a & = & b\cdot a\\
	b ~\tilde{*}~ a & = & b*a,
	\end{array}\]
	i.e. the new set of actions is the same with the old one.
\end{corollary}

\begin{proposition}
	Let $(A,B,\alpha)$ be a crossed module in $\C$ and $d\in \D_{\alpha}(B,A)$. Then $(A,B,\alpha_d)$ is a crossed module in $\C$ with the new set of derived actions obtained from $d$ where $\alpha_d=\sigma_{d}^{-1}\alpha$.
\end{proposition}

\begin{proof}
	We need to prove that $\alpha_d$ satisfies the conditions \textbf{(CM1)-(CM4)} of Proposition \ref{xmodprop}.
	\begin{enumerate}[label=\textbf{(CM\arabic{*})}, leftmargin=1.7cm]
		\item Let $a\in A$ and $b\in B$. Then,
		\[\begin{array}{rl}
		\alpha_d\left( {^b}a \right)   & = \alpha_d(\sigma_{d}(b)\cdot a) \\
		& = \sigma_{d}^{-1}\alpha(\sigma_{d}(b)\cdot a)\\
		& = \sigma_{d}^{-1}(\sigma_{d}(b)+\alpha(a)-\sigma_{d}(b)) \\
		& = \sigma_{d}^{-1}(\sigma_{d}(b))+\sigma_{d}^{-1}(\alpha(a))-\sigma_{d}^{-1}(\sigma_{d}(b)) \\
		& = b + \alpha_d(a) - b.
		\end{array}\]
		\item Let $a,a_1\in A$. Then,
		\[\begin{array}{rl}
		{^{\alpha_d(a)}}a_1   & = \sigma_d(\alpha_d(a))\cdot a_1 \\
		& = \sigma_d(\sigma_{d}^{-1}\alpha(a))\cdot a_1 \\
		& = \alpha(a)\cdot a_1 \\
		& = a+a_1-a.
		\end{array}\] 
		\item Let $a\in A$ and $b\in B$. Then,
		\[\begin{array}{rl}
		\alpha_d\left(b ~\tilde{*}~ a \right)   & = \sigma_{d}^{-1}\alpha\left(\sigma_{d}(b) * a \right)\\
		& = \sigma_{d}^{-1}(\sigma_{d}(b) * \alpha\left(a \right)) \\
		& = \sigma_{d}^{-1}(\sigma_{d}(b))*\sigma_{d}^{-1}(\alpha\left(a \right)) \\
		& = b*\alpha_d(a)
		\end{array}\]
		and similarly
		\[\alpha_d\left(a ~\tilde{*}~ b \right) = \alpha_d(a)*b. \]
		\item Let $a,a_1\in A$. Then,
		\[\begin{array}{rl}
		{\alpha_d(a)} ~\tilde{*}~ a_1   & = \sigma_{d}^{-1}\alpha(a) ~\tilde{*}~ a_1\\
		& = \sigma_d(\sigma_{d}^{-1}\alpha(a)) * a_1 \\
		& = \alpha(a) * a_1 \\
		& = a*a_1
		\end{array}\] 
		and similarly
		\[a ~\tilde{*}~ {\alpha_d(a_1)} = a*a_1. \]
	\end{enumerate}
	Hence $(A,B,\alpha_d)$ is a crossed module in $\C$ with the new set of derived actions obtained from $d$.
\end{proof}

We call this new crossed module $(A,B,\alpha_d)$ a \textit{derived crossed module} from $d\in\D_{\alpha}(B,A)$ and denote it by $d(A,B,\alpha)$.

\begin{corollary}
	Let $(A,B,\alpha)$ be a crossed module in $\C$. Then for the zero derivation $0\in \D_{\alpha}(B,A)$, $\alpha_{0}=\alpha$, i.e. $(A,B,\alpha)$ and $0(A,B,\alpha)=(A,B,\alpha_{0})$ are the same crossed modules.
\end{corollary}

\begin{corollary}
	Let $(A,B,\alpha)$ be a crossed module in $\C$ and $d\in \D_{\alpha}(B,A)$. Then \[(1,\sigma_d^{-1})\colon (A,B,\alpha)\rightarrow (A,B,\alpha_d)\]
	is a morphism (isomorphism) of crossed modules in $\C$.
\end{corollary}

Ak{\i}z et al. \cite[p.236]{Ak13} defined a crossed module $(\w{A},\w{B},\w{\alpha})$ to be a covering of a crossed module $(A,B,\alpha)$ if there is a crossed module morphism $(f_1,f_0)\colon(\w{A},\w{B},\w{\alpha})\rightarrow(A,B,\alpha)$ such that $f_1\colon \w{A}\rightarrow A$ is an isomorphism in $\C$.

\begin{corollary}
	Let $(A,B,\alpha)$ be a crossed module in $\C$ and $d\in \D_{\alpha}(B,A)$. Then $(A,B,\alpha)$ is a covering of $(A,B,\alpha_d)$.
\end{corollary}

With the technique given in \cite[3.10. Corollary]{Sa18} we now construct a regular derivation of $(A,B,\alpha_d)$ over $\sigma_d$ using a regular derivation of $(A,B,\alpha)$.

\begin{proposition}\label{newder}
	Let $(A,B,\alpha)$ be a crossed module in $\C$ and $d\in \D_{\alpha}(B,A)$. Then $d'=d\sigma_d\in \D_{\alpha_d}(B,A)$, i.e. $d'=d\sigma_d$ is a regular derivation of $(A,B,\alpha_d)$.
\end{proposition}

\begin{proof}
	First we need to prove that $d'=d\sigma_d$ satisfies the conditions of Definition \ref{der}.
	\begin{enumerate}[label=\textbf{(\roman{*})}, leftmargin=1.5cm]
		\item Let $b,{{b}_{1}}\in B$. Then
		\[\begin{array}{rl}
		d'(b+b_1)   & = d\sigma_{d}(b+b_1)\\
		& = d(\sigma_{d}(b)+\sigma_{d}(b_1)) \\
		& = d(\sigma_{d}(b))+{\sigma_d(b)}\cdot d(\sigma_{d}(b_1)) \\
		& = d'(b)+{^{b}}d'(b_1).
		\end{array}\] 
		\item Let $b,{{b}_{1}}\in B$ and $*\in\Omega_{2}'$. Then
		\[\begin{array}{rl}
		d'(b*b_1)   & = d\sigma_{d}(b*b_1)\\
		& = d(\sigma_{d}(b)*\sigma_{d}(b_1)) \\
		& = d\sigma_{d}(b)*d\sigma_{d}(b_1)+d\sigma_{d}(b)*\sigma_{d}(b_1)+\sigma_{d}(b)*d\sigma_{d}(b_1)  \\
		& = d'(b)*d'(b_1)+d'(b) ~\tilde{*}~ b_1+b ~\tilde{*}~ d'(b_1) \\
		& = d'(b*b_1).
		\end{array}\] 
		\item Let $b,{{b}_{1}}\in B$ and $*\in\Omega_{2}'$. Then
		\[\begin{array}{rl}
		d'(\om (b))   & = d\sigma_{d}(\om (b))\\
		& = \om(d\sigma_{d} (b)) \\
		& = \om (d'(b)).
		\end{array}\]
	\end{enumerate}
	Hence $d'=d\sigma_d$ is a derivation of $(A,B,\alpha_d)$. Now we need to prove that $d'$ is regular. The endomorphism $\theta_{d'}\colon A\rightarrow A$ is defined with
	\[\begin{array}{rl}
	\theta_{d'}(a)  & = d'\alpha_d(a)+a\\
	& = d\sigma_d\sigma_{d}^{-1}\alpha(a)+a \\
	& = d\alpha(a)+a \\
	& = \theta_d(a)
	\end{array}\]
	for all $a\in A$. Since $d\in \D_{\alpha}(B,A)$ then by Proposition \ref{ucluprop}, $\theta_d=\theta_{d'}$ is an automorphism of $A$. And again by Proposition \ref{ucluprop}, $d'\in\D_{\alpha_d}(B,A)$. This completes the proof.
\end{proof}

By Proposition \ref{newder}, for all $d\in\D_{\alpha}(B,A)$ we obtain a sequence of isomorphic crossed modules as in the following.

\[\xymatrix@=2.7pc{
	(A,B,\alpha) \ar[r]^-{(1,\sigma_d^{-1})}_-{\cong} &  (A,B,\alpha_d) \ar[r]^-{(1,\sigma_{d'}^{-1})}_-{\cong} & (A,B,\alpha_{d'}) \ar[r]^-{(1,\sigma_{d''}^{-1})}_-{\cong} & (A,B,\alpha_{d''}) \ar[r]^-{(1,\sigma_{d'''}^{-1})}_-{\cong} & \cdots }\]
or equally
\[\xymatrix@=2.7pc{
	(A,B,\alpha) \ar[r]^-{(1,\sigma_d^{-1})}_-{\cong} &  d(A,B,\alpha) \ar[r]^-{(1,\sigma_{d'}^{-1})}_-{\cong} & d'(d(A,B,\alpha)) \ar[r]^-{(1,\sigma_{d''}^{-1})}_-{\cong} & d''(d'(d(A,B,\alpha))) \ar[r]^-{(1,\sigma_{d'''}^{-1})}_-{\cong} & \cdots }\]

%
%

\section{Conclusion}

Whitehead defined the notion of derivations as a special morphisms for crossed modules over groups. In this paper, in order to define homotopy of crossed module morphisms and derivations of crossed modules in an arbitrary category $\C$ of groups with operations which is a more general case according to works of Whitehead \cite{W48} and Norrie \cite{Nor90}, we follow the method used by Cockroft \cite{Co52} and Norrie \cite{Nor90}. Moreover, we obtain new crossed modules on the same objects in $\C$ using the derivations of old one.

\end{document}